\documentclass[12pt]{article}

\usepackage[a4paper]{geometry}
\usepackage{amsmath,amssymb,amsthm}
\usepackage{graphicx}
\usepackage{tikz}
\usetikzlibrary{arrows}
\usetikzlibrary{calc}
\usetikzlibrary{intersections}
\usetikzlibrary{arrows,shapes,positioning}
\usetikzlibrary{decorations.markings}
\tikzstyle arrowstyle=[scale=1]
\usetikzlibrary{graphs}
\tikzstyle{directed}=[postaction={decorate,decoration={markings,mark=at position .7 with {\arrow[arrowstyle]{stealth}}}}]
\usepackage[colorlinks=true,citecolor=black,linkcolor=black,urlcolor=blue]{hyperref}

\usepackage[all]{xy}
\usepackage{enumerate}

\usepackage{diagbox}
\usepackage{comment}
\DeclareMathOperator{\R}{\mathbb{R}}

\DeclareMathOperator{\Mt}{\mathcal{M}}                     

\newtheorem{definition}{Definition}[section]

\newtheorem{example}{Example}[section]
\newtheorem{theorem}{Theorem}[section]
\newtheorem{lemma}{Lemma}[section]
\newtheorem{proposition}{Proposition}[section]

\newtheorem{remark}{Remark}[section]

\title{Graphical zonotopes with the same face vector}

\author{Zeying Xu
\thanks{\small Department of Mathematics, Shanghai Jiao Tong University,  Shanghai, China.
E-mail: \nobreak{zeying\_xu@outlook.com}}
}

\begin{document}

\maketitle

\begin{abstract}
We are interested in constructing zonotopes which are combinatorially nonequivalent but have the same face vector.  In this paper we introduce a quadrilateral flip operation on graphs. We show that, if one graph is obtained from another graph by a flip, then the face vectors of the graphical zonotopes of these two graphs are the same. In this way, we can easily construct a class of  combinatorially nonequivalent graphical zonotopes which share the same face vector. It is known that all triangulations of the $n$-gon are connected through the flip operation. Thus their graphical zonotopes have the same face vector. We will compute this vector and the total number of faces.

 \end{abstract}

\section{Introduction}

A zonotope is the Minkowski sum of several line segments.
Given a set of vectors $V=\{v_1,\dots,v_m\}$ in $\R^n$, we define a zonotope
\[Z(V)=[0,v_1]+\dots + [0,v_m].\]
The combinatorial structure of faces of $Z(V)$ is totally determined by the oriented matroid $\Mt(V)$ of $V$: the face poset of $Z(V)$ is anti-isomorphic with the poset of covectors of $\Mt(V)$ \cite{BLSWZ99}. In this paper we are interested in constructing and understanding  zonotopes which are combinatorially different but have the same face vector.

Let $G=(V(G),E(G))$ be a simple (no loops or multiple
edges) connected graph on vertex set $\{1,2,\dots,n\}$. The graphical zonotope $Z_G$ of $G$ is defined by
\[Z_G=\sum_{ij\in E(G)}[e_i, e_j],\]
where $e_1,\dots,e_n$ are the coordinate vectors in $\R^n$. Many properties of graph $G$ are encoded by its graphical zonotope. For example, the volume of $Z_G$ equals the number of spanning trees of $G$ and the number of lattice points in $Z_G$ equals to the number of forests in $G$ \cite[Proposition 2.4]{postnikov09}. The number of vertices of $Z_G$ is equal to the number of acyclic orientations of $G$.
Graphical zonotope of the complete graph is just the permutahedron and all graphical zonotopes are in the class of  generalized permutahedron \cite{postnikov09}.
 Gruji\'c \cite{Grujic17} showed a relation between the face vector of $Z_G$ and  the q-analog of the chromatic symmetric function of $G$. Study of face vectors of graphical zonotopes can also be found in Postnikov et.al. \cite{PRW08}.

For each edge $ij$ of $G$, we denote by $\overrightarrow{{ij}}$ the orientation of $ij$ from $i$ to $j$. A partial orientation $X$ of $G$ is an orientation of a subgraph $H$ of $G$. We can regard $X$ as the directed graph whose underlying graph is $H$. If all the edges of $G$ are oriented, then we say $X$ is a full orientation. The {\em genus} $g(X)$ of $X$ is defined to be the genus $g(H)$ of $H$, which is  the edge number of $H$ minus the vertex number of $H$ and then plus the number of connected components of $H$. Let $X^0$ denote the subgraph of $G$ with edge set $E(G)\setminus E(H)$.  For two partial orientations $X, Y$ of $G$, we define $X\leq Y$ if $X$ is a sub-digraph of $Y$. Two partial orientations $X$ and $Y$ are {\em orthogonal} if either no edge is simultaneously oriented by $X$ and $Y$ or there exist at least two edges such that one is oriented by $X$ and $Y$ in a same direction and the other is oriented by  $X$ and $Y$ in reverse directions.

For a graph $G$, we let $r(G)$ be the edge number of a maximal spanning forest of $G$. In the following, we introduce the oriented matroid of graph $G$ in terms of partial orientations.

 A partial orientation  $X$ of  $G$ is a {\em vector of $G$} if every arc (oriented edge) of $X$ is contained in a directed circuit. Let $\mathcal V(G)$ denote the set of vectors of graph $G$. Then the poset $(\mathcal V(G),\leq)$ is graded and the rank $rank(X)$ of each vector $X$ is given by $g(X)$.

 A partial orientation  $X$ of  $G$ is a {\em covector of $G$} if it is orthogonal with every directed circuit of $G$. Let $\mathcal L(G)$ denote the set of covectors of graph $G$.  Then the poset $(\mathcal L(G),\leq)$ is graded and the rank $rank(X)$ of each covector $X$ is given by $r(G)-r(X^0)$ (see \cite[Corollary 4.1.15]{BLSWZ99}).

Given a planar graph $G$, let $G^*$ be the dual planar graph of $G$. Then  $(\mathcal V(G^*),\leq)$ is isomorphic with $(\mathcal L(G),\leq)$.

\begin{proposition} \label{prop:anti}
The face poset of $Z_G$ is anti-isomorphic with the poset $(\mathcal L(G), \leq)$.
\end{proposition}

This paper is organized as follows. In Section \ref{sec:flip} we introduce the quadrilateral flip operation on graphs and prove that graphical zonotopes of two flip-equivalent graphs have the same face vector. In Section \ref{sec:gon} we compute the face vector and total face number of graphical zonotopes of triangulations of the $n$-gon.
\section{Flip on graphs}
\label{sec:flip}

Flipping has long been important in the study of
triangulations. Given a triangulation of a point configuration in the plane, a flip means replacing two triangles by two different triangles that cover the same quadrilateral. The flip graph is a graph defined on triangulations such that two triangulations are connected if they are related by a flip. Please refer to  \cite{bose09} for a survey of the flip graph of triangulations of planar point set. A generalization of flip operation to higher dimension is called bistellar flip \cite{de10}.

We define the following quadrilateral flip operation on graphs.

\begin{definition}[quadrilateral flip]
Let $G$ be a simple connected graph. Let $v_1,v_2,v_3,v_4$ be four  vertices of $G$ such that the subgraph $H$ induced by these four vertices is a $4$-cycle with edges $v_1v_2,v_2v_3,v_3v_4,v_4v_1$ and a diagonal edge $v_2v_4$.  Suppose further that we can divide $V(G)\setminus \{v_1,v_2,v_3,v_4\}$ into four disjoint sets $V_1,V_2,V_3,V_4$  such that, for every edge $uv$ of $G$ which is not $v_2v_4$, $u$ and $v$ belong to  $V_i\cup\{v_i,v_{i+1}\}$ ($v_5$ is $v_1$) for some $i\in \{1,2,3,4\}$. Then let $G'$ be the new graph obtained from $G$ by deleting edge $v_2v_4$ and then adding edge $v_1v_3$. Such an operation from $G$ to $G'$ is called a {\em quadrilateral flip}, or simply a {\em flip}.
\end{definition}

See figure below for illustration of the flip operation. We say two graphs are flip-equivalent if one can be transformed to another through a sequence of flip operations.

\begin{figure}[htbp]
\centering

\begin{tikzpicture}[inner sep=0pt, v/.style={circle,draw=black,fill=black, inner sep=0.6pt}]
\node[v](v) at (-0.5,-0.5) [label=below left:$v_4$] {};
\node[v](a) [above  =of v, label=above left:$v_1$] {};
\node[v](b) [ right =of v, label=below right:$v_3$] {};
\node[v](u) [right =of a, label=above right:$v_2$] {};



\draw[-] (u) to (a) to (v) to (b) to (u) to (v);

\draw[-,dashed] (a) to [out=90,in=180] (0,1.5) to [out=0,in=90]  (u);
\draw[-,dashed] (u) to [out=0,in=90]  (1.5,0) to [out=-90,in=0] (b);
\draw[-,dashed] (b) to [out=-90,in=0] (0,-1.5) to [out=-180,in=-90](v);
\draw[-,dashed] (v) to [out=180,in=-90] (-1.5,0)to [out=90,in=180] (a);

\node at (3,0) [label=above:flip]{$\longrightarrow$};

\tikzset{xshift=6cm}

\node[v](v) at (-0.5,-0.5) [label=below left:$v_4$] {};
\node[v](a) [above  =of v, label=above left:$v_1$] {};
\node[v](b) [ right =of v, label=below right:$v_3$] {};
\node[v](u) [right =of a, label=above right:$v_2$] {};



\draw[-] (a) to (v) to (b) to (u) to (a) to (b);

\draw[-,dashed] (a) to [out=90,in=180] (0,1.5) to [out=0,in=90]  (u);
\draw[-,dashed] (u) to [out=0,in=90]  (1.5,0) to [out=-90,in=0] (b);
\draw[-,dashed] (b) to [out=-90,in=0] (0,-1.5) to [out=-180,in=-90](v);
\draw[-,dashed] (v) to [out=180,in=-90] (-1.5,0)to [out=90,in=180] (a);

\end{tikzpicture}
\end{figure}

\begin{theorem}\label{thm:flip}
Graphical zonotopes of two flip-equivalent graphs have the same face vector.
\end{theorem}

\begin{proof}
Let $G'$ be a graph obtained from a graph $G$ by a flip that removes  edge $v_2v_4$ and adds edge $v_1v_3$. It is obvious that $r(G)=r(G')$. For each $i=1,2,3,4$, let $G_i$ be the induced subgraph of $G$ with vertex set $V_i\cup \{v_i,v_{i+1}\}$. Let $H$ and $H'$ be the induced subgraphs of $G$ and $G'$ with vertex set $\{v_1,v_2,v_3,v_4\}$, respectively. In what follows, we will explicitly construct a bijection from covectors $X$ of $G$ to covectors $Y$ of $G'$ such that $rank(X)=rank(Y)$. Note that $Y$ is a covector of $G'$ if its restrictions on $G_1,\dots,G_4$ and $H'$ are covectors of $G_1,\dots,G_4$ and $H'$, respectively.  Let $X'$ and $Y'$ be the restriction of $X$ and $Y$ on  $H$ and $H'$, respectively.

\textbf{Case 1.} No edge in $H$ is oriented by $X$. Let $Y$ equals $X$ on edges not in $H'$ and let $Y$ do not orient any edge in $H'$. It is obvious that $rank(X)=rank(Y)$.

\begin{figure}[htbp]
\centering

\begin{tikzpicture}[ inner sep=0pt, v/.style={circle,draw=black,fill=black, inner sep=0.6pt}]

\node[v](v) at (-0.5,-0.5) [label=below left:$v_4$] {};
\node[v](a) [above  =of v, label=above left:$v_1$] {};
\node[v](b) [ right =of v, label=below right:$v_3$] {};
\node[v](u) [right =of a, label=above right:$v_2$] {};

\draw[-] (u) to (a) to (v) to (b) to (u) to (v);
\draw[-,directed] (u) to (a);
\draw[-,directed] (v) to (a);

\node at (1,0) {$\longrightarrow$};

\node at (1,-1.3) {$(1)$};

\tikzset{xshift=2cm}

\node[v](v) at (-0.5,-0.5) [label=below left:$v_4$] {};
\node[v](a) [above  =of v, label=above left:$v_1$] {};
\node[v](b) [ right =of v, label=below right:$v_3$] {};
\node[v](u) [right =of a, label=above right:$v_2$] {};

\draw[-] (a) to (v) to (b) to (u) to (a) to (b);
\draw[-,directed] (u) to (a);
\draw[-,directed] (v) to (a);
\draw[-,directed] (b) to (a);


\tikzset{xshift=2cm}

\node[v](v) at (-0.5,-0.5) [label=below left:$v_4$] {};
\node[v](a) [above  =of v, label=above left:$v_1$] {};
\node[v](b) [ right =of v, label=below right:$v_3$] {};
\node[v](u) [right =of a, label=above right:$v_2$] {};

\draw[-] (u) to (a) to (v) to (b) to (u) to (v);
\draw[-,directed] (u) to (a);
\draw[-,directed] (u) to (b);
\draw[-,directed] (u) to (v);

\node at (1,0) {$\longrightarrow$};
\node at (1,-1.3) {$(2)$};

\tikzset{xshift=2cm}

\node[v](v) at (-0.5,-0.5) [label=below left:$v_4$] {};
\node[v](a) [above  =of v, label=above left:$v_1$] {};
\node[v](b) [ right =of v, label=below right:$v_3$] {};
\node[v](u) [right =of a, label=above right:$v_2$] {};

\draw[-] (a) to (v) to (b) to (u) to (a) to (b);
\draw[-,directed] (u) to (a);
\draw[-,directed] (u) to (b);

\tikzset{xshift=2cm}

\node[v](v) at (-0.5,-0.5) [label=below left:$v_4$] {};
\node[v](a) [above  =of v, label=above left:$v_1$] {};
\node[v](b) [ right =of v, label=below right:$v_3$] {};
\node[v](u) [right =of a, label=above right:$v_2$] {};

\draw[-] (u) to (a) to (v) to (b) to (u) to (v);
\draw[-,directed] (u) to (a);
\draw[-,directed] (b) to (v);
\draw[-,directed] (u) to (v);

\node at (1,0) {$\longrightarrow$};
\node at (1,-1.3) {$(3)$};

\tikzset{xshift=2cm}

\node[v](v) at (-0.5,-0.5) [label=below left:$v_4$] {};
\node[v](a) [above  =of v, label=above left:$v_1$] {};
\node[v](b) [ right =of v, label=below right:$v_3$] {};
\node[v](u) [right =of a, label=above right:$v_2$] {};

\draw[-] (a) to (v) to (b) to (u) to (a) to (b);
\draw[-,directed] (u) to (a);
\draw[-,directed] (b) to (v);
\draw[-,directed] (b) to (a);

\end{tikzpicture}
\caption{Case 2.}
\label{fig:casetwo}
\end{figure}

\textbf{Case 2.} $X$ orients two edges in the $4$-cycle of $H$. This case has three subcases: $(1)$, two edges incident with $v_1$ or $v_3$ are oriented; $(2)$, two edges incident with $v_2$ or $v_4$ are oriented, then $v_2v_4$ is also oriented; $(3)$, two non-adjacent edges are oriented, then $v_2v_4$ is also oriented. We have shown one example for each of the three cases in Fig. \ref{fig:casetwo}. For each case, we let $Y$ equals $X$ on all edges of $G'$ other than $v_1v_3$. Then the orientation situation of $Y$ on edge $v_1v_3$ is determined. For example, in figure $(1)$, since $\overrightarrow{v_2v_1},\overrightarrow{v_4v_1}\in Y$, we must have $\overrightarrow{v_3v_1}\in Y$ to make $Y$  a covector. It is easy to see that $Y$ is indeed a covector of $G'$. Moreover, we have $rank(X)=rank(Y)$. Indeed, for example, in case $(1)$, $X^0$  can be obtained from $Y^0$ by adding one more edge between two vertices in a same connected component of $Y^0$. So $r(X^0)=r(Y^0)$, and thus $rank(X)=rank(Y)$.

\textbf{Case 3.} $X$ orients three edges in the $4$-cycle of $H$. See Fig. \ref{fig:casethree} for an example. In this case $u_2u_4$ is also oriented. If not, then there is a triangle with only one edge oriented, contradicting the fact that $X$ is a covector. We let $Y$ equals $X$ on edges of $G'$ other than $u_1u_3$. Note that there is also a triangle of $H'$ containing edge $u_1u_3$ and exactly one of the two boundary edges is oriented. So let $Y$ orient edge $u_1u_3$ in the only allowed direction. Now $X^0=Y^0$, so $rank(X)=rank(Y)$.

\begin{figure}[htbp]
\centering

\begin{tikzpicture}[ inner sep=0pt, v/.style={circle,draw=black,fill=black, inner sep=0.6pt}]

\node[v](v) at (-0.5,-0.5) [label=below left:$v_4$] {};
\node[v](a) [above  =of v, label=above left:$v_1$] {};
\node[v](b) [ right =of v, label=below right:$v_3$] {};
\node[v](u) [right =of a, label=above right:$v_2$] {};

\draw[-] (u) to (a) to (v) to (b) to (u) to (v);
\draw[-,directed] (u) to (a);
\draw[-,directed] (b) to (v);
\draw[-,directed] (u) to (v);
\draw[-,directed] (b) to (u);

\node at (1.5,0) {$\longrightarrow$};

\tikzset{xshift=3cm}

\node[v](v) at (-0.5,-0.5) [label=below left:$v_4$] {};
\node[v](a) [above  =of v, label=above left:$v_1$] {};
\node[v](b) [ right =of v, label=below right:$v_3$] {};
\node[v](u) [right =of a, label=above right:$v_2$] {};

\draw[-] (a) to (v) to (b) to (u) to (a) to (b);
\draw[-,directed] (u) to (a);
\draw[-,directed] (b) to (v);
\draw[-,directed] (b) to (u);
\draw[-,directed] (b) to (a);
\end{tikzpicture}
\caption{Case 3.}
\label{fig:casethree}
\end{figure}

\textbf{Case 4.} $X$ orients four edges in the $4$-cycle of $H$. In this case we construct $Y$ in the following way. We first let $Y'$ be isomorphic with $X'$ under the map $v_i\rightarrow v_{i+1}$. So if forgetting about labels, then $Y'$ looks just like the rotation of $X'$ by $90$ degree. For each $i=1,2,3,4$, let $X_i$ and $Y_i$ be the restriction of $X$ and $Y$ on $G_i$ respectively. Then for each $i=1,2,3,4$, if $v_iv_{i+1}$ is oriented in the same direction for $X'$ and $Y'$, set $Y_i=X_i$; otherwise, let $Y_i$ be the reverse of $X_i$, denoted as $-X_i$. See figure below for illustration. Though in  example below $v_2v_4$ is oriented in $X$, we also have cases that $v_2v_4$ is not oriented.

\begin{figure}[htbp]
\centering

\begin{tikzpicture}[inner sep=0pt, v/.style={circle,draw=black,fill=black, inner sep=0.6pt}]

\node[v](v4) at (-0.5,-0.5) [label=below left:$v_4$] {};
\node[v](v1) [above  =of v4, label=above left:$v_1$] {};
\node[v](v3) [ right =of v4, label=below right:$v_3$] {};
\node[v](v2) [above right =of v4, label=above right:$v_2$] {};

\draw[-,directed] (v1) to (v2);
\draw[-,directed] (v1) to (v4);
\draw[-,directed] (v2) to (v4);
\draw[-,directed] (v2) to (v3);
\draw[-,directed] (v3) to (v4);

\draw[-,dashed] (v1) to [out=90,in=180] (0,1.5) to [out=0,in=90]  (v2);
\draw[-,dashed] (v2) to [out=0,in=90]  (1.5,0) to [out=-90,in=0] (v3);
\draw[-,dashed] (v3) to [out=-90,in=0] (0,-1.5) to [out=-180,in=-90](v4);
\draw[-,dashed] (v4) to [out=180,in=-90] (-1.5,0)to [out=90,in=180] (v1);

\node at (0,1) {$X_1$};
\node at (1,0) {$X_2$};
\node at (0,-1) {$X_3$};
\node at (-1,0) {$X_4$};

\node at (3,0) {$\longrightarrow$};

\tikzset{xshift=6cm}

\node[v](v4) at (-0.5,-0.5) [label=below left:$v_4$] {};
\node[v](v1) [above  =of v4, label=above left:$v_1$] {};
\node[v](v3) [ right =of v4, label=below right:$v_3$] {};
\node[v](v2) [above right =of v4, label=above right:$v_2$] {};

\draw[-,directed] (v4) to (v1);
\draw[-,directed] (v4) to (v3);
\draw[-,directed] (v1) to (v3);
\draw[-,directed] (v1) to (v2);
\draw[-,directed] (v2) to (v3);

\draw[-,dashed] (v1) to [out=90,in=180] (0,1.5) to [out=0,in=90]  (v2);
\draw[-,dashed] (v2) to [out=0,in=90]  (1.5,0) to [out=-90,in=0] (v3);
\draw[-,dashed] (v3) to [out=-90,in=0] (0,-1.5) to [out=-180,in=-90](v4);
\draw[-,dashed] (v4) to [out=180,in=-90] (-1.5,0)to [out=90,in=180] (v1);

\node at (0,1) {$X_1$};
\node at (1,0) {$X_2$};
\node at (0,-1) {$-X_3$};
\node at (-1,0) {$-X_4$};

\end{tikzpicture}
\end{figure}

It is easy to see that $Y$ is a covector of $G'$. Now we will show that $rank(X)=rank(Y)$. Note that for each $i\in \{1,2,3,4\}$,  there is no path from $v_i$ to $v_{i+1}$ in $X_i^0$. If not, together with edge $v_iv_{i+1}$ we have a cycle with only one  edge oriented, contradicting the fact that  $X$ is a covector. So we see that $r(X^0)=r(Y^0)=r(X_1^0)+\dots+r(X_4^0)$ when $v_2v_4$ is oriented and $r(X^0)=r(Y^0)=r(X_1^0)+\dots+r(X_4^0)+1$ when $v_2v_4$ is not oriented. So $rank(X)=rank(Y)$.

At last, through above cases we have defined a map from covectors $X$ of $G$ to covectors $Y$ of $G'$. We can see that this map is injective. As $G$ and $G'$ are flip-equivalent, we can also define a similar map from $G'$ to $G$. Thus our map is a bijection. We have already shown that this map keeps rank. So the theorem is proved.
\end{proof}

\begin{remark}
In our definition of flip on graphs, the extra condition on $V_1,\dots,V_4$ is essential.
For example, the following two graphs are skeleton graphs of two triangulations of $5$ points on the plane that are related by a flip. But these two graphs are not flip-equivalent in our definition. Their graphical zonotopes do have different face vectors. For the left graph it is $(72,150,102,24)$ and for the right graph it is $(78,168,116,26)$. The computation is done by Polymake \cite{polymake}.
\begin{figure}[htbp]
\centering
\begin{tikzpicture}[scale=0.7, inner sep=0pt, v/.style={circle,draw=black,fill=black, inner sep=0.6pt}]

\draw
(0,0) --++(1,1) --++ (2,-1) --++ (-1,0) --++ (-1,1) --++  (0,-2) --++ (2,1);

\draw[-] (0,0) to (1,-1);
\draw[-] (1,-1) to (2,0);

\tikzset{xshift= 6cm}

\draw
(0,0) --++(1,1) --++ (2,-1) --++ (-3,0) --++ (1,-1) --++ (2,1);

\draw[-] (2,0) to (1,1);
\draw[-] (2,0) to (1,-1);

\end{tikzpicture}
\end{figure}

\end{remark}

\begin{remark}
Whitney \cite{whitney33} characterized when two graphs have isomorphic matroids. The matroids of two flip-equivalent graphs are usually not isomorphic. Thus their graphical zonotopes are usually not combinatorially equivalent.
\end{remark}

\begin{example}
The following kind of graphs admit several flip operations. These graphs are in tree shapes; namely, we can  divide such a graph into blocks and then the adjacency relation between these blocks is a tree.

\begin{figure}[htbp]
\centering

\begin{tikzpicture}[scale=0.5, inner sep=0pt, v/.style={circle,draw=black,fill=black, inner sep=0.6pt}]
\def \x{4cm};
\def \y{2cm};

\node[v](v4) at (-0.5,-0.5)  {};
\node[v](v1) at (-0.5,0.5) {};
\node[v](v3) at (0.5,-0.5) {};
\node[v](v2) at (0.5,0.5) {};

\draw[-] (v1) to (v2) to (v3) to (v4) to (v1) to (v4) to (v2);

\node[v](m4) at ($(-0.5,-0.5)+(0.5*\x,\y)$)  {};
\node[v](m1) at ($(-0.5,0.5)+(0.5*\x,\y)$) {};
\node[v](m3) at ($(0.5,-0.5)+(0.5*\x,\y)$) {};
\node[v](m2) at ($(0.5,0.5)+(0.5*\x,\y)$) {};

\draw[-] (m1) to (m2) to (m3) to (m4) to (m1) to (m4) to (m2);

\node[v](u4) at ($(-0.5,-0.5)+(\x,0)$)  {};
\node[v](u1) at ($(-0.5,0.5)+(\x,0)$) {};
\node[v](u3) at ($(0.5,-0.5)+(\x,0)$) {};
\node[v](u2) at ($(0.5,0.5)+(\x,0)$) {};

\draw[-] (u1) to (u2) to (u3) to (u4) to (u1) to (u3);

\node[v](w4) at ($(-0.5,-0.5)+2*(\x,0)$)  {};
\node[v](w1) at ($(-0.5,0.5)+2*(\x,0)$) {};
\node[v](w3) at ($(0.5,-0.5)+2*(\x,0)$) {};
\node[v](w2) at ($(0.5,0.5)+2*(\x,0)$) {};

\draw[-] (w1) to (w2) to (w3) to (w4) to (w1) to (w4) to (w2);

\draw[-,dashed]
(m4) to (v2)
(v3) to (u4)
(u1) to (m3)
(u2) to (w1)
(u3) to (w4)
;

\draw[-,dashed] (v1) to [out=90,in=180] ($0.5*(v1)+0.5*(v2)+(0,1)$) to [out=0,in=90]  (v2);
\draw[-,dashed] (v3) to [out=-90,in=0] ($0.5*(v3)+0.5*(v4)+(0,-1)$) to [out=-180,in=-90](v4);
\draw[-,dashed] (v4) to [out=180,in=-90] ($0.5*(v1)+0.5*(v4)+(-1,0)$)to [out=90,in=180] (v1);

\draw[-,dashed] (w1) to [out=90,in=180] ($0.5*(w1)+0.5*(w2)+(0,1)$) to [out=0,in=90]  (w2);
\draw[-,dashed] (w2) to [out=0,in=90]  ($0.5*(w3)+0.5*(w2)+(1,0)$) to [out=-90,in=0] (w3);
\draw[-,dashed] (w3) to [out=-90,in=0] ($0.5*(w3)+0.5*(w4)+(0,-1)$) to [out=-180,in=-90](w4);

\draw[-,dashed] (u1) to [out=90,in=180] ($0.5*(u1)+0.5*(u2)+(0,1)$) to [out=0,in=90]  (u2);
\draw[-,dashed] (u3) to [out=-90,in=0] ($0.5*(u3)+0.5*(u4)+(0,-1)$) to [out=-180,in=-90](u4);

\draw[-,dashed] (m1) to [out=90,in=180] ($0.5*(m1)+0.5*(m2)+(0,1)$) to [out=0,in=90]  (m2);
\draw[-,dashed] (m2) to [out=0,in=90]  ($0.5*(m3)+0.5*(m2)+(1,0)$) to [out=-90,in=0] (m3);
\draw[-,dashed] (m4) to [out=180,in=-90] ($0.5*(m1)+0.5*(m4)+(-1,0)$)to [out=90,in=180] (m1);

\end{tikzpicture}
\end{figure}
\end{example}

\section{Triangulations of the $n$-gon}
\label{sec:gon}

The $n$-gon is a $2$-dimensional polytope with $n$ edges.
In this section we consider special graphs that are  skeleton graphs of triangulations of the $n$-gon. From definition, it is clear that if two triangulations of the $n$-gon are related by a flip, then their skeleton graphs are flip-equivalent.
 It is well known that the flip graph of triangulations of the $n$-gon can be realized as the skeleton graph of the $n-2$ dimensional associahedron (also called the Stasheff polytope) \cite{Lee89}.  The diameter of this flip graph is $2n-8$ when $n> 11$ \cite{pournin14,sleator88}. See Fig. \ref{fig:fg6} for an example of flip graph of triangulations of the $6$-gon.

\begin{figure}[htbp]
\centering
\includegraphics[scale=0.3]{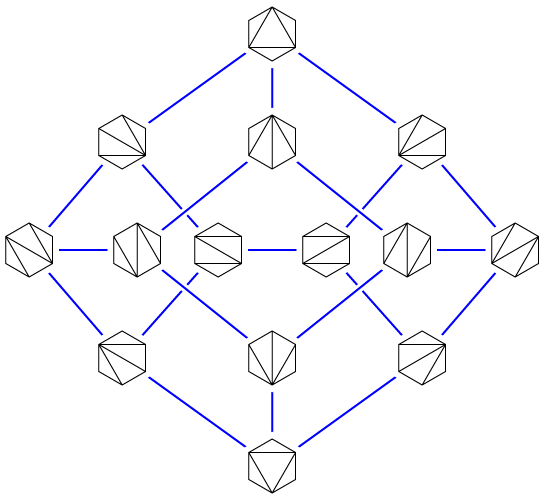}
\caption{Flip graph of triangulations of the $6$-gon \cite{fg6}.}
\label{fig:fg6}
\end{figure}

\begin{proposition}\label{prop:connectivity}
The flip graph of triangulations of the $n$-gon is connected.
\end{proposition}

In the following, we do not distinguish a triangulation of the $n$-gon and its skeleton graph.  Combining  Theorem \ref{thm:flip}  and Proposition \ref{prop:connectivity}, we immediately see that graphical zonotopes of all triangulations of the $n$-gon have a same face vector. We will prove  the following results in this section.
\begin{theorem} \label{thm:facevector}
The graphical zonotope of every triangulation of the $n$-gon has the face vector $f=(f_0,f_1,\dots,f_{n-1})$ where for $i=0,\dots,n-1:$
\[f_{n-1-i}=\sum_{m=0}^{\min\{i,n-i\}}2^m3^{i-m}{i-1\choose m-1}{n \choose m+i}.\]
\end{theorem}

\begin{theorem}\label{thm:totalnumber}
The total number of faces of the graphical zonotope of a triangulation of the $(n+2)$-gon is $\left(
 \begin{array}{cc}
 1 & 1 \\
 \end{array}
 \right)
\left(
      \begin{array}{cc}
        1 & 1 \\
        2 & 4 \\
      \end{array}
    \right)^{n}
  \left(
    \begin{array}{c}
      1 \\
      2 \\
    \end{array}
  \right)$, which is also equal to

\noindent $3\sum_{\frac{n}{2}\leq m\leq n}{m\choose n-m}5^{2m-n}(-2)^{n-m}-2\sum_{\frac{n-1}{2}\leq m\leq n-1}{m\choose n-m-1}5^{2m-n+1}(-2)^{n-m-1}$.
\end{theorem}

Let $T_n$ be the triangulation of the $n$-gon on vertex set $\{1,2,\dots,n\}$ such that all its interior edges are incident with vertex $1$. Note that $T_n$ is a planar graph. Its dual planar graph $T_n^*$ can be constructed from the {\em binary caterpillar tree}  $C_n$ by gluing all leaves into one vertex, where $C_n$ is a tree such that all its interior vertices have degree three and are on a path. We assume the leaves of $C_n$ are labelled by $1,2,\dots,n$ from left to right. See Fig.~\ref{fig:dual} for an example of $T_5, T_5^*$ and $C_5$.

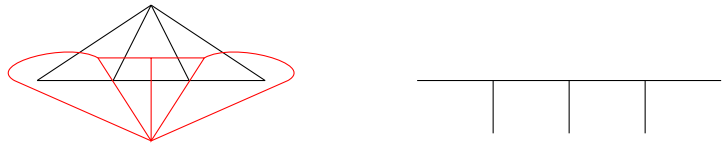
\begin{figure}[htbp]
\centering

\begin{tikzpicture}[v/.style={circle,draw=black,fill=black, inner sep=0.6pt}]

\coordinate (r) at (-0.5,0);

\foreach \i in {-2,-1,0,1}
{
\coordinate (v\i) at (\i,-1);

\draw[-] (r) to (v\i);
}

\draw[-] (v1) to (v-2);

\coordinate (t) at (-0.5,-1.8);
\coordinate (u1) at (-1.2, -0.7);
\coordinate (u2) at (-0.5,-0.7);
\coordinate (u3) at (0.2,-0.7);

\draw[-,color=red] (t) to  (u1) to (u2) to (u3) to (t) to (u2);

\draw[-,color=red] (u1) to [out=150, in=150] (-2.3, -1) to  (t);
\draw[-,color=red] (u3) to [out=30, in=30] (1.3, -1) to (t);

\draw[-] (3,-1) to (7,-1);
\draw[-] (5,-1) to (5,-1.7);
\draw[-] (6,-1) to (6,-1.7);
\draw[-] (4,-1) to (4,-1.7);

\end{tikzpicture}
\caption{$T_5, T^*_5$ and $C_5$.}
\label{fig:dual}
\end{figure}

As $T_n^*$ is obtained from $C_n$ by gluing its leaves, we can naturally identify partial orientations of $T_n^*$ with partial orientations of $C_n$. So we let $\mathcal V(C_n)$ be the set of partial orientations of $C_n$ that correspond to vectors of $T_n^*$, and also call elements in $\mathcal V(C_n)$ {\em vectors} of $C_n$. Then we have the following isomorphism
\[ (\mathcal L(T_n),\leq)\cong (\mathcal V(T_n^*),\leq)\cong (\mathcal V(C_n),\leq).\]
 
The following observation is obvious.
\begin{proposition}
A partial orientation $X$ of $C_n$ is a vector if and only if in $X$ every arc  is on a directed path between two leaves of $C_n$.
\end{proposition}

Now we investigate the rank function of poset $(\mathcal V(C_n),\leq)$. For each partial orientation $X$ of $C_n$, let $\Gamma_X$ be the digraph on leaves of $C_n$ such that $\overrightarrow{ij}\in \Gamma_X$ if $i$ can reach $j$ through a directed path in $X$. Let $b_0(\Gamma_X)$ be the number of weakly connected components of $\Gamma_X$.

\begin{lemma}\label{lem:rank}
The rank of an element $X$ in $(\mathcal V(C_n),\leq)$ is $n-b_0(\Gamma_X)$.
\end{lemma}

\begin{proof}
Let $X'$ be the covector of $T_n^*$ that corresponds to covector $X$ of $C_n$. Let $H$ be the underlying graph of $X'$. Then we know that $rank(X)=rank(X')=g(H)$.

Suppose that $\Gamma_X$ have $k$ isolated vertices and $m$ weakly connected components with vertex sizes $c_1,\dots,c_m$ which are at least $2$.  Then the underlying graph of $X$ is a disjoint union of $m$ subtrees $T_1,\dots, T_m$ of $C_n$. For each tree $T_i$, let $H_i$ be obtained from $T_i$ by gluing its leaves, then $g(H_i)=c_i-1$. Observe that $H$ can be obtained from $H_1,\dots, H_m$ by taking one vertex from each graph and then gluing these $m$ vertices into one vertex. So $g(H)=g(H_1)+\dots+g(H_m)=c_1+\dots+c_m-m=n-k-m=n-b_0(\Gamma_X)$.
\end{proof}

\begin{lemma} \label{lem:vertex}
The number of full vectors of $C_n$ is $2\cdot 3^{n-2}$ when $n\geq 2$.
\end{lemma}

\begin{proof}
The subgraph obtained from $C_n$ by deleting leaves $n$ and $n-1$ is isomorphic with $C_{n-1}$. The restriction of every full vector of $C_n$ on $C_{n-1}$ gives rise to a full vector of $C_{n-1}$. On the other hand, every full vector of $C_{n-1}$ can be extended to $3$ different full vectors of $C_n$. Also note that $C_2$ has $2$ full vectors. So the number of full vectors of $C_n$ is $2\cdot 3^{n-2}$.
\end{proof}

\begin{proof}[Proof of Theorem \ref{thm:facevector}]
In the following, we will compute the vector $f$ basing on $Z_{T_n}$. Let $i$ be an integer in $\{0,1,\dots, n-1\}$. According to Proposition \ref{prop:anti}, $f_{n-1-i}$ is the number of covectors of $T_n$ of rank $i$, which is also the number of vectors of $C_n$ of rank $i$. Then by Lemma \ref{lem:rank}, we conclude that  $f_{n-1-i}$ is the number of vectors $X$ of $C_n$ such that $b_0(\Gamma_X)=n-i$.

For each integer $m\leq n-i$, we enumerate the number of vectors $X$ of $C_n$ such that  $b_0(\Gamma_X)=n-i$ and exactly $m$ connected components of $\Gamma_X$ are not isolated vertices. Let $L$ be the set of isolated vertices of $\Gamma_X$. Then $|L|=n-i-m$.   Let $T$ be the subtree of  $C_n$ such that leaves of $T$ are leaves of $C_n$ that are not in $L$, and let $T'$ be obtained from $T$ by contracting all degree $2$ vertices. Then $T'$ is isomorphic with $C_{i+m}$.  Note that if $v$ is an interior vertex of $C_n$ that is adjacent with a leaf in $L$, let $uv$ and $wv$ be another two edges incident with $v$, then either $uv, wv\in X^0$ or $\overrightarrow{uv}, \overrightarrow{vw}\in X$ or $\overrightarrow{wv}, \overrightarrow{vu}\in X$.  So we let $X'$ be the vector of $T'$ such that each edge $ab$ is oriented from $a$ to $b$ whenever  $a$ can reach $b$ in $X$. Now $b_0(\Gamma_{X'})= m$ and $\Gamma_{X'}$ has no isolated vertices. So $X'^0$ must be $m-1$ interior edges of $T'$ such that no two of them are adjacent. Let $T'_1,\dots, T'_m$ be the connected components of the underlying graph of $X'$. We assume their leaf numbers are $c_1,\dots,c_m$ respectively. For each $i=1,\dots,m$,  let $X'_i$ be the restriction of $X'$ on $T'_i$, then $X'_i$ is a full vector of $T'_i$.

From above analysis we see that $X$ is totally determined by data $L$, $X'^0$ and $X'_1,\dots,X'_{m}$. We have ${n\choose m+i}$ choice for $L$. To choose $X'^0$ from $T'$ is equivalent with choosing $m-1$ non-adjacent numbers from $\{1,2,\dots,i+m-3\}$, and the  number of choices to this classical combinatorial problem is ${(i+m-3)-(m-1)+1 \choose m-1}={i-1\choose m-1}$. After $L$ and $X'^0$ are fixed, then $T'_1,\dots, T'_m$ are also fixed. Their leaf numbers satisfy $c_1+\dots +c_m=m+i$. By Lemma \ref{lem:vertex}, the number of choices of $X'_i$ is $2\cdot 3^{c_i-2}$ for each $i=1,\dots, m$. So the total number of choice of $X'_1,\dots,X'_{m}$ is $\prod_{i=1,\dots, m} 2\cdot 3^{c_i-2}=2^m 3^{c_1+\dots +c_m-2m}=2^m 3^{i-m}$.

To sum up, the number of vectors $X$ of $C_n$ such that  $b_0(\Gamma_X)=n-i$ and exactly $m$ connected components of $\Gamma_X$ are not isolated vertices is
\[2^m 3^{i-m}{n\choose m+i}{i-1\choose m-1}.\]
Thus $f_{n-i-1}=\sum_{0\leq m\leq \min\{i,n-i\}}2^m 3^{i-m}{n\choose m+i}{i-1\choose m-1}$.
\end{proof}

To enumerate the total number of faces of the graphical zonotope of a triangulation of the $n$-gon, we find it more convenient to do recursion than applying the formula in Theorem \ref{thm:facevector}.

\begin{proof}[Proof of Theorem \ref{thm:totalnumber}]
Let $\mathcal V_n^0,\mathcal V_n^1,\mathcal V_n^{-1}$ be the sets of vectors of $C_n$ such that the pendent edge of leaf $n$ is not oriented, oriented to $n$, oriented from $n$ respectively. Let $\mathcal V_n$ be the set of all  vectors of $C_n$.

Let $v$ be the interior vertex of $C_n$ which is adjacent with leaves $n$ and $n-1$. Let $u$ be another vertex of $C_n$ which is adjacent with $v$. Let $T$ be the subgraph of $C_n$ obtained by deleting leaves $n$ and $n-1$. Note that $T$ is isomorphic with $C_{n-1}$. Every vector $X$ of $C_n$ defines a vector $X'$ of $T$ which is the  restriction of $X$ on $T$.

For each vector $X\in \mathcal V_n^0$, if $vn-1,vu\in X^0$ then $X'\in \mathcal V_{n-1}^0$; if $\overrightarrow{uv}, \overrightarrow{vn-1}\in X$ then $X'\in \mathcal V_{n-1}^+$; if $\overrightarrow{vu}, \overrightarrow{n-1v}\in X$ then $X'\in \mathcal V_{n-1}^-$. So we see that
\begin{equation}\label{eq:v0}
   |\mathcal V^0_n| =  |\mathcal V^0_{n-1}|+ |\mathcal V^+_{n-1}|+ |\mathcal V^-_{n-1}|=|\mathcal V_{n-1}|.
\end{equation}

Similar analysis for $\mathcal V_n^+$ and $\mathcal V_n^-$ shows that
\begin{equation}\label{eq:v+}
    |\mathcal V^+_n| =  |\mathcal V^0_{n-1}|+ 3|\mathcal V^+_{n-1}|+ |\mathcal V^-_{n-1}|
\end{equation}
 and
 \begin{equation}\label{eq:v-}
     |\mathcal V^-_n| =  |\mathcal V^0_{n-1}|+ |\mathcal V^+_{n-1}|+ 3|\mathcal V^-_{n-1}|.
 \end{equation}

Denote $\mathcal V^*_n= \mathcal V^+_n\cup \mathcal V^-_n$. Then we have
\begin{equation*}
  \left(
    \begin{array}{c}
      |\mathcal V^0_n| \\
      |\mathcal V^*_n| \\
    \end{array}
  \right)
  = \left(
      \begin{array}{cc}
        1 & 1 \\
        2 & 4 \\
      \end{array}
    \right)
  \left(
    \begin{array}{c}
      |\mathcal V^0_{n-1}| \\
      |\mathcal V^*_{n-1}| \\
    \end{array}
  \right).
\end{equation*}
Note that $(|\mathcal V^0_2|,|\mathcal V^*_2|)=(1,2)$, so we have
\begin{equation*}
  \left(
    \begin{array}{c}
      |\mathcal V^0_{n+2}| \\
      |\mathcal V^*_{n+2}| \\
    \end{array}
  \right)
  = \left(
      \begin{array}{cc}
        1 & 1 \\
        2 & 4 \\
      \end{array}
    \right)^{n}
  \left(
    \begin{array}{c}
      1 \\
      2 \\
    \end{array}
  \right).
\end{equation*}
So $|\mathcal V_{n+2}|=\left(
 \begin{array}{cc}
 1 & 1 \\
 \end{array}
 \right)
\left(
      \begin{array}{cc}
        1 & 1 \\
        2 & 4 \\
      \end{array}
    \right)^{n}
  \left(
    \begin{array}{c}
      1 \\
      2 \\
    \end{array}
  \right).$

If we take the summation of Eqs.\eqref{eq:v0},  \eqref{eq:v+} and \eqref{eq:v-}, we will also obtain the recursive formula
\[|\mathcal V_n|=5|\mathcal V_{n-1}|-2|\mathcal V_{n-2}|.\]
By method of generating function, we then get
\[|\mathcal V_{n+2}|=3\sum_{\frac{n}{2}\leq m\leq n}{m\choose n-m}5^{2m-n}(-2)^{n-m}-2\sum_{\frac{n-1}{2}\leq m\leq n-1}{m\choose n-m-1}5^{2m-n+1}(-2)^{n-m-1}.\]
\end{proof}

\begin{remark}
The  sequence of total face number $|\mathcal V_{n}|$ is the sequence A052984 in OEIS \cite{oeis}. They are  also the Kekul\'{e} numbers for certain benzenoids \cite[see p. 78]{cyvin13}.
\end{remark}

\begin{figure}[htbp]\label{fig:Lynx}
\centering
\begin{tabular}{|l|llllllll|l|}
  \hline
  \diagbox[dir=SE]{$n$}{$i$} & 0 & 1& 2& 3& 4& 5 & 6 & 7 & Total\\
  \hline
  2 & 2 & 1 &  &  &  &  &  & & 3\\
  3 & 6 & 6 & 1 &  &  &  &  & & 13\\
  4 & 18 & 28 & 12 & 1 &  &  & & & 59\\
  5 & 54 & 114 & 80 & 20 & 1 &  &  & &  269\\
  6 & 162 & 432 & 422 & 180 & 30 & 1 &  &  & 1227\\
  7 & 486 & 1566 & 1962 & 1190 & 350 & 42 & 1 & & 5597\\
  8 & 1458 & 5508 & 8424 & 6640 & 2828 & 616 & 56 & 1 & 25531\\
  \hline
\end{tabular}
\caption{The face vectors of graphical zonotopes of triangulations of $n$-gon with $n=2, \ldots, 8$.}
\end{figure}

\bibliographystyle{plain}
\bibliography{reference}
\end{document}